\documentclass[12pt]{amsart}
\usepackage{amsmath, amsthm, amssymb}

\newtheorem{theorem}{Theorem}[section]
\newtheorem{lemma}[theorem]{Lemma}

\newtheorem{ex}{Exercise}[section]
\newtheorem{examp}[theorem]{Example}
\newtheorem{corollary}[theorem]{Corollary}
\newtheorem{remar}[theorem]{Remark}
\newcommand{\diams}{\unskip\nobreak\hfil\penalty50%
\hskip1em\hbox{}\nobreak\hfil%
$\diamondsuit$\parfillskip=0pt\finalhyphendemerits=0}

\newcommand{\cal}{\mathcal}
\newcommand{\bfind}[1]{\index{#1}{\bf #1}}

\newcommand{\n}{\par\noindent}

\newcommand{\sn}{\par\smallskip\noindent}
\newcommand{\mn}{\par\medskip\noindent}
\newcommand{\bn}{\par\bigskip\noindent}
\newcommand{\pars}{\par\smallskip}
\newcommand{\parm}{\par\medskip}
\newcommand{\parb}{\par\bigskip}
\newcommand{\subsetuneq}{\mathrel{\raisebox{.8ex}{\footnotesize%
$\displaystyle\mathop{\subset}_{\not=}$}}}
\newcommand{\isom}{\simeq}

\newcommand{\cB}{\mathcal B}
\newcommand{\bcB}{\overline{\mathcal B}}
\newcommand{\cO}{\mathcal O}

\newcommand{\cX}{\mathcal X}
\newcommand{\cZ}{\mathcal Z}
\font\tenlv=msbm10 scaled 1200
\font\sevenlv=msbm7 scaled 1200
\font\fivelv=msbm5 scaled 1200
\def\lv #1{{\mathchoice{{\hbox{\tenlv #1}}}{{\hbox{\tenlv #1}}}
{{\hbox{\sevenlv #1}}}{{\hbox{\fivelv #1}}}}}

\newcommand{\N}{\lv N}
\newcommand{\Q}{\lv Q}
\newcommand{\R}{\lv R}

\newcommand{\Z}{\lv Z}
\begin{document}
\title[Valuation theory of exponential Hardy fields II]{Valuation theory of exponential Hardy fields II: Principal parts of germs in the Hardy field of o-minimal exponential expansions of the reals}
\author{Franz-Viktor Kuhlmann and Salma Kuhlmann}
\address{Institute of Mathematics,
University of Silesia,
ul.~Bankowa 14,
40-007 Katowice,
Poland}
\email{fvk@math.us.edu.pl}
\address{FB Mathematik und Statistik, Universit\"at Konstanz, 78464 Konstanz, Germany}
\email{salma.kuhlmann@uni-konstanz.de}
\date{12.\ 9.\ 2016}
\begin{abstract}
We present a general structure theorem for the Hardy field of an o-minimal expansion of the reals by restricted analytic functions and an unrestricted exponential.
We proceed to analyze its residue fields with respect to arbitrary convex valuations, and deduce a power series expansion of exponential germs. We apply these results to cast ``Hardy's conjecture'' (see \cite[p.111]{[KS]}) in a more general framework.
This paper is a follow up to \cite{[K--K2]} and is partially based on unpublished results of \cite{[K--K]}. A previous version \cite{[K--K1]}
(which was dedicated to Murray A. Marshall on his 60th birthday) remained unpublished. In \cite{[W]} our structure theorem for the residue fields was rediscovered and applied to the diophantine context. Due to this revived interest, we decided to rework the arXiv preprint \cite{[K--K1]} and submit it to the Marshall Memorial Volume.
\end{abstract}
\maketitle

%
%
\section{Introduction}\label{section1}
In this paper, we analyze the structure of the Hardy fields associated
with o-minimal expansions of the reals with exponential function. More precisely, we take $T$ to be
the theory of a polynomially bounded o-minimal expansion ${\cal P}$ of
the ordered field of real numbers by a set ${\cal F}_T$ of real-valued
functions. We assume that the language of $T$ contains a symbol for
every 0-definable function, and that $T$ defines the
restricted exponential and logarithmic functions. Now let $T(\exp)$ denote the
theory of the expansion $({\cal P},\exp)$ where $\exp$ is the
un-restricted real exponential function.
Then also $T(\exp)$ is o-minimal, and admits
quantifier elimination and a universal axiomatization in the language
augmented by $\log$ \cite{[D--S2]}. We consider the Hardy field $H({\cal P},\exp)$ (see
Section~\ref{sectHf} for the definition). Our general assumptions (see Section \ref{sectrf}) \ imply that $H({\cal P},\exp)$ is a model of $T(\exp)$ and is equal to the
closure $LE_{\cal F _T}(x)$ of its subfield ${\mathbb R}(x)$ under real closure, ${\cal F _T}\,$, $\exp$ and
its inverse $\log$; here, $x$ denotes the germ of the identity function \cite{[D--M--M2]}.
\mn
We study convex valuations on $H({\cal P},\exp)$. To this end, for ${\cal F}\subseteq {\cal F}_T$, we introduce an intrinsic form of power
series expansions for the elements of $LE_{\cal F}(x)$.
We use monomials, which are the elements in the image of a suitable
cross-section, together with coefficients from residue fields $LE_{\cal F}(x)w$ with respect to significant convex valuations $w$.
We apply our results in particular to ${\cal F} = {\cal F}_{\rm an}$ (the family of restricted analytic functions),  $T=T_{\rm an}$  (the polynomially bounded o-minimal  theory of the expansion $\mathbb{R}_{\rm an}$ of the reals by restricted analytic functions), see \cite{[D--M--M2]} for more details about this theory.
\mn
The paper is organised as follows. In Section \ref{sectprel} we gather in a concise manner the necessary background. In Section \ref{section3} we prove our structure theorem for $LE_{\cal F}(x)$ and its residue fields, see Theorem \ref{mth1}. The main result leading to the definition of principal parts (see the definition in Section \ref{section4}) is Theorem \ref{thrprh}. Section \ref{section4} is dedicated to its proof. The final Section \ref{section5} considers applications to the Hardy field $H(\R_{\rm an, exp})$.
The principal part of a function $h\in H(\R_{\rm an, exp})$ carries
information about the asymptotic behavior of the function $\exp h(x)$
(Theorem~\ref{asppth}). Corollary~\ref{asppcor} gives a powerful criterion - using principal parts - for an exponential germ to be asymptotic to a composition of semialgebraic functions, exp, log and restricted analytic functions.  This puts the particular solution of the Hardy
problem in a more general framework; see the computations following Corollary~\ref{asppcor}.
Finally, we provide a further application to embeddings of Hardy fields into fields of generalized power series, see Corollary \ref{embps}.

%
%
\section{Some preliminaries}                \label{sectprel}
%
%
%
\subsection{Valuations}
If $(K,w)$ is a valued field, then we write $wa$ for the value of $a\in
K$ and $wK$ for its value group $\{wa\mid 0\ne a\in K\}$. Further, we
write $aw$ for the residue of $a$, and $Kw$ for the residue field. The
valuation ring is denoted by ${\cal O}_w\,$. For generalities on
valuation theory, see \cite{[R]}, and for convex valuations in particular see \cite{[KS]} or \cite{[EP]}.

A valuation $w$ on an ordered field $K$ is called \bfind{convex} if
${\cal O}_w$ is convex. The convex valuation rings of an ordered field
are linearly ordered by inclusion. If ${\cal O}_w\subsetuneq {\cal
O}_{w'}$ then $w$ is said to be \bfind{finer} than $w'$, and $w'$ is a 
\bfind{coarsening} of $w$. If $w$ and $w'$ are two convex valuations on the same ordered field, we will write $w<w'$ if $w$ is a proper
coarsening of $w'$, that is, if $\cO_{w'} \subsetuneq \cO_w\,$.

\mn

There is always
a finest convex valuation, called the \bfind{natural valuation}. It is
characterized by the fact that its residue field is archimedean. A
valuation $w$ on an ordered field is convex if and only if the natural
valuation is finer or equal to $w$.
{\it Throughout this paper, $v$ will always denote the natural
valuation,} unless stated otherwise.

If $a,b$ are elements of an ordered group or an ordered field, then we
write $a\ll b<0$ if $a<b<0$ and $\forall n\in\N: a<nb$. Similarly,
$a\gg b>0$ if $a>b>0$ and $\forall n\in\N: a>nb$. We set $|a|:=\max
\{a,-a\}$. Then the natural valuation is characterized by:
\begin{equation}                            \label{prelnv}
va<vb\>\Leftrightarrow\> |a|\gg |b|\;.
\end{equation}
Note that if $\R\subset K$ and $a\in K$ with $va=0$, then there is some
$r\in\R$ such that $v(a-r)>0$. Further, $wr=0$ for every $r\in\R$
and every convex valuation $w$.

\begin{lemma}                               \label{prelcoars}
Let $v,w$ be arbitrary valuations on some field $K$. Suppose that $v$
is finer than $w$. Then for all $a,b\in K$,
\begin{equation}                            \label{prelfin}
va\leq vb\>\Rightarrow\>wa\leq wb\;.
\end{equation}
In particular, $wa>0\Rightarrow va>0$. Further, $H_w:=\{vz\mid z\in
K\wedge wz=0\}$ is a convex subgroup of the value group $vK$ of $v$. We
have that $vz\in H_w\Leftrightarrow z\in {\cal O}_w^\times\,$. There is
a canonical isomorphism $wK\isom vK/H_w\,$. Conversely, every convex
subgroup of $vK$ is of the form $H_w$ for some valuation $w$ such that
$v$ is finer or equal to $w$.

The valuation $v$ of $K$ induces a valuation $v/w$ on $Kw$. There are
canonical isomorphisms $v/w(Kw)\isom H_w$ and $(Kw)v/w\isom Kv$. If
$Kw$ is embedded in ${\cal O}_w$ such that the restriction of the
residue map is the identity on $Kw$, then $v/w=v|_{Kw}$
(up to equivalence). Writing $v$ instead of $v|_{Kw}\,$, we then
have that $v(Kw)=H_w$ and $(Kw)v=Kv$.
\end{lemma}
We will call $H_w$ the \bfind{convex subgroup associated with $w$}
and $w$ the \bfind{valuation associated with $H_w\,$}. Since
the isomorphism is canonical, we will write $wK=vK/H_w\,$.

The order type of the chain of nontrivial convex subgroups of an ordered
abelian group $G$ is called the \bfind{rank} of $G$. If finite, then the
rank is not bigger than the maximal number of rationally independent
elements in $G$ (which is the dimension of its divisible hull as a $\Q$-vector space). In particular, $G$ has finite rank if it is finitely generated.

\pars
From (\ref{prelnv}) and (\ref{prelfin}) it follows that for every convex
valuation $w$,
\begin{equation}                            \label{prelnv1}
|a|\leq |b|\>\Rightarrow\> wa\geq wb\;.
\end{equation}

\bn

Take any valued field $(K,v)$. A \bfind{field of representatives for the residue field of $(K,v)$} is a subfield
$k$ of $K$ such that $v$ is trivial on $k$ (or equivalently, $k$ is contained in the valuation ring $\cO$), and
for every $a\in\cO$ there is $b\in k$ such that $v(a-b)>0$. It then follows that the residue map $\cO\ni a\mapsto
av$ induces an isomorphism from $k$ to the residue field. A \bfind{cross-section} of $(K,v)$ is an embedding $\iota$
of the value group $vK$ in the multiplicative group $K^\times$ such that $v\iota(\alpha)=\alpha$ for all $\alpha\in vK$.

%
%
\subsection{Hardy fields}                   \label{sectHf}
Let us recall some basic facts about Hardy fields (see Chapter 6, Section 2 in \cite{[KS]}).  
Assume that $T$ is the theory of any
o-minimal expansion  ${\cal R}$ of the ordered field of real numbers by real-valued
functions. The Hardy field of
${\cal R}$, denoted by $H({\cal R})$, is the set of germs at $\infty$ of
unary ${\cal R}$-definable functions $f:{\mathbb R}\rightarrow {\mathbb R}$.
Then $H({\cal R})$ is an ordered differential field which contains
${\cal R}$ as a substructure. Let $x\in H({\cal R})$ be the germ of the identity function.
Then $H({\cal R})$ is the closure of ${\mathbb R}(x)$ under all 0-definable
functions of ${\cal R}$, \cite{[D--M--M2]}.

If $f,g$ are non-zero unary
${\cal R}$-definable functions on ${\cal R}$, then we will denote their
germs in $H({\cal R})$ by the same letters. The
following holds for non-zero germs:
\begin{equation}
v f\>=\>v g \;\Longleftrightarrow\;\lim_{x\rightarrow
\infty}\frac{f(x)}{g(x)}\;\mbox{ is a non-zero constant in }{\mathbb R}\;.
\end{equation}
The non-zero germs $f$ and $g$ are \bfind{asymptotic}
 if and only if this constant is $1$, and we have:
\begin{equation}                                      \label{asylog}
\mbox{ $f$ and $g$ are asymptotic} \;\Longleftrightarrow\; v(f-g)\> >\> v(g)\>.
\end{equation}

see \cite[Lemma 6.22]{[KS]} 

%
%
\subsection{ General assumptions on $T$}
\label{sectrf}
 Throughout this paper, we will assume
that $T$ is the theory of a
polynomially bounded o-minimal expansion ${\cal P}$ of the ordered field
of real numbers by real-valued functions. Further, we assume that $T$
defines the restricted exp and log. Then also $T(\exp)$ is o-minimal
(cf.\ \cite{[D--S2]}). Here, $T(\exp)$ denotes the theory of the expansion
$({\cal P},\exp)$ where $\exp$ is the un-restricted real exponential
function. 

We let ${\cal F}_T$ denote the set of function symbols in the language
of $T$ and assume that there is a function symbol in ${\cal F}_T$ for
each 0-definable function of ${\cal P}$. This implies that $T$ admits
quantifier elimination and a universal axiomatization. We let ${\cal F}$ denote any subset of ${\cal F}_T\,$.

We denote by $M$ a model of $T$. Often, we will assume further
that $M$ is a model of $T(\exp)$ (but will not distinguish notationally between
$M$ and its reduct to the language of $T$.) Suppose that the field $K$
is a submodel (and hence elementary submodel) of $M$. Take $x_i\in M$,
$i\in I$. By $K\langle x_i\mid i\in I\rangle$ we denote the
$0$-definable closure of $K\cup\{x_i\mid i\in I\}$ in $M$. By our
assumption on the language of $T$, it is the closure of $K\cup\{x_i\mid
i\in I\}$ under ${\cal F}_T\,$, that is, the smallest subfield of $M$
containing $K\cup \{x_i\mid i\in I\}$ and closed under all functions
which interpret the function symbols of ${\cal F}_T$ in $M$. Since $T$
admits a universal axiomatization and $K\langle x_i\mid i\in I\rangle$
is a substructure of $M$, it is a model of $T$. Since $T$ admits
quantifier elimination, $K\langle x_i\mid i\in I\rangle$ is an
elementary substructure of $M$.

%

For an arbitrary subfield $F\subseteq M$, the real closure $F^{\rm r}$ of
$F$ can be taken to lie in $M$ since $M$ is real closed. We denote by
$F^{\rm h}$ the henselization of $(F,v)$. It can be taken to lie in $M$
since the natural valuation $v$ of the real closed field $M$ is henselian.

We let $F^{\cal F}$ denote the smallest subfield of $M$ which contains
$F$ and is ${\cal F}$-closed, that is, closed under all functions on $M$
which are interpretations of function symbols in ${\cal F}$. Analogously, we define $F^{\rm h\cal F}$ to be
the smallest subfield of $M$ which contains $F$ and is ${\cal F}$-closed
and henselian w.r.t.\ $v$, and $F^{\rm r\cal F}$ to be the smallest such subfield which is in addition real closed.
Note that $F^{\cal F}\subseteq F^{\rm h\cal F} \subseteq F^{\rm r\cal F}$.

%
%
\section{A general structure theorem for $LE_{\cal F}(x)$}\label{section3}
In what follows, we work under the assumptions of \cite [Lemma 6.40; pp. 104-105]{[KS]}. More precisely, we let $M$ be a model of $T=T_{\rm an}$ (or of $T_{\rm an}(\exp)$, and $\cal F \subset \cal F_{\rm an}$ be an arbitrary set of convergent power series representing restricted analytic functions, closed under partial derivatives, and containing the restricted $\exp$ and $\log$.
\mn

For the proof  Theorem~\ref{mth1} below,  we need the following lemma. 

\begin{lemma}                               \label{basic2}
Let $M$ be a model of $T_{\rm an}$, $x_i\in M$ be such that the values $vx_i\,$, $i\in I$ are rationally independent. Further, let $w$ be any
convex valuation. Assume that there is a subset $I_w\subset I$ such that $wx_i=0$ for all $i\in I_w$ and that the
values $wx_i\,$, $i\in I\setminus I_w$ are rationally independent.
Then
\[w\R(x_i\mid i\in I)^{\rm r\cal F}\>=\>
\bigoplus_{i\in I\setminus I_w}\Q  wx_i\;\mbox{ \ and \ }\;
w\R(x_i\mid i\in I)^{\rm h\cal F}\>=\>
\bigoplus_{i\in I\setminus I_w} \Z wx_i\;.\]
Further,
\[
\R(x_i\mid i\in I_w)^{\rm r\cal F}
\]
is a field of representatives for the residue field $\R(x_i\mid i\in I)^{\rm r\cal F}w$, and
\[
\R(x_i\mid i\in I_w)^{\rm h\cal F}
\]
is a field of representatives for the residue field $\R(x_i\mid i\in I)^{\rm h\cal F}w$.

\pars
Assume in addition that all $x_i$ with $i\in I\setminus I_w$ are positive. Then the multiplicative group of
$\R(x_i\mid i\in I\setminus I_w)^{\rm r \cal F}$ contains the divisible hull $\cX$ of the group generated by all of these
$x_i$, $\cX$ is the image of a suitably chosen cross-section, and the following holds:
\begin{equation}                     \label{h=r}
\R(x_i\mid i\in I_w)^{\rm r\cal F}(\cX)^{\rm r}\>=\>\R(x_i\mid i\in I_w)^{\rm r\cal F}(\cX)^{\rm h}\>.
\end{equation}
\end{lemma}
\begin{proof}
The first part of this lemma is \cite[Lemma 6.40]{[KS]}.

Now assume that all $x_i >0$ for all $i\in I\setminus I_w$. Since $\R(x_i\mid i\in I\setminus I_w)^{\rm r \cal F}$  is real closed, it
contains $x_i^{1/k}$ for all $i\in I\setminus I_w$ and $k\in\N$. This yields that its multiplicative group contains the divisible hull $\cX$ of the group generated
by all of these $x_i\,$.

The restriction of $w$ to $\cX$ is a group homomorphism onto the value group $w\R(x_i\mid i\in I)^{\rm r\cal F}$;
it is injective since the values $wx_i\,$, $i\in I\setminus I_w$ are rationally independent. The inverse of this
isomorphism is a cross-section with image $\cX$.

From what we have proved, we obtain that $w\R(x_i\mid i\in I_w)^{\rm r\cal F}(\cX)^{\rm h}=
w\R(x_i\mid i\in I_w)^{\rm r\cal F}(\cX)=w\R(x_i\mid i\in I)^{\rm r\cal F}$, which is divisible. Further,
the residue field of $\R(x_i\mid i\in I_w)^{\rm r\cal F}(\cX)^{\rm h}$ is $\R(x_i\mid i\in I_w)^{\rm r\cal F}$,
which is real closed. Thus by \cite[Theorem 4.3.7]{[EP]},
$\R(x_i\mid i\in I_w)^{\rm r\cal F}(\cX)^{\rm h}$ is real closed, which gives equation~(\ref{h=r}).
\end{proof}

 We now fix any non-archimedean model
$M$ of $T(\exp)$ which contains $(\R,+,\cdot ,<,{\cal F},\exp)$
as a substructure. We recall from the introduction that $LE_{\cal F}(x)$
denotes the closure of the subfield ${\mathbb R}(x)$ under real closure, ${\cal F}$, $\exp$ and
its inverse $\log$; here, $x$ denotes any infinitely large and positive element (i.e.\ $x>0$ and $vx < 0$).
The following is the structure theorem which we will put to work. 

\begin{theorem}                             \label{mth1}
$LE_{\cal F}(x)$ is of the form
\begin{equation}                            \label{formR}
\R(\cX)^{\rm r\cal F}\>=\>\R(\cX)^{\rm h\cal F}\>,
\end{equation}
where $\cX$ is a subgroup of the multiplicative group of positive elements of $LE_{\cal F}(x)$ which is the image
of a cross-section, with the following properties:
\sn
a) \ $\cX$ contains $x$ and $\log_m x$ for all $m\in\N$,
\sn
b) \ for every convex valuation $w$ on $LE_{\cal F}(x)$, if
\[
\cX_w\>:=\>\{x'\in \cX\mid wx'=0\}\>,
\]
then
\begin{equation}                            \label{rf=R}
\R(\cX_w)^{\rm r {\cal F}}\>=\> \R(\cX_w)^{\rm h {\cal F}}\>\subseteq\> LE_{\cal F}(x)
\end{equation}
is a field of representatives for the residue field $LE_{\cal F}(x)w$. Identifying $LE_{\cal F}(x)w$ with this
field of representatives, we obtain that
\begin{equation}                            \label{ww'}
LE_{\cal F}(x)w\>\subseteq\>LE_{\cal F}(x)w'\>\subseteq\>LE_{\cal F}(x)
\end{equation}
for all coarsenings $w$ of $v$ and $w'$ of $w$.
\end{theorem}

Note that the set $\cX$ is not uniquely determined. However, we will fix it throughout this paper and call
the elements of $\cX$ the \bfind{monomials} of $LE_{\cal F}(x)$. Correspondingly, we fix the residue fields
$LE_{\cal F}(x)w=\R(\cX_w)^{\rm h {\cal F}}$ for all convex valuations $w$ on $LE_{\cal F}(x)w$.
\begin{proof}
According to \cite[Theorem 6.30]{[KS]}, $LE_{\cal F}(x)$ is of the form
\begin{equation}                            \label{wform}
\R(x_i\mid i\in I)^{\rm r\cal F}\;\mbox{ with $x_i>0$ and $vx_i$ rationally independent,}
\end{equation}
and $x$ and $\log_m x$, $m\in\N$, among the $x_i\,$. Applying Lemma~\ref{basic2} with $w=v$, we find that the
multiplicative group of $LE_{\cal F}(x)$ contains the divisible hull $\cX$ of the subgroup generated by the
$x_i\,$, and that $\cX$ is the image of a cross-section. With $I_v=\emptyset$, we further obtain that
$\R(\cX)^{\rm r \cal F}=\R(\cX)^{\rm h \cal F}$, which implies equation (\ref{formR}).

\parm
It remains to prove part b).
Take a convex valuation $w$ on $LE_{\cal F}(x)$. The group $\cX$ is isomorphic to the divisible value group
$vLE_{\cal F}(x)$, so it is a $\Q$-vector space. For $\cX_w=\{x'\in \cX\mid wx'=0\}$, the values $v\cX_w$ form
a convex subgroup of this value group, which consequently is also divisible and a $\Q$-vector space. Hence also
$\cX_w$ is a $\Q$-vector space. We choose a basis $\cB_w$ of $\cX_w$ and a basis $\cB'_w$ of a complement of
$\cX_w$ in $\cX$. We write $\cB_w=\{x_i\mid i\in I_w\}$, $\cB'_w=\{x_i\mid i\in I'_w\}$ and set $I=I_w\cup I'_w\,$.
As $\{x_i\mid i\in I\}$ is a basis of $\cX$ which is isormorphic to the value group through the valuation, the
values $vx_i$, $i\in I$, are rationally independent. Further, the elements $x_i$, $i\in I\setminus I_w=I'_w$ are
$\Q$-linearly independent over the $\Q$-vector space $\cX_w\,$, which means that no nontrivial linear combination
of these elements has value 0 under $w$. In other words, the values $wx_i$, $i\in I\setminus I_w\,$, are rationally
independent.

\pars
Now we apply Lemma~\ref{basic2} to obtain that $\R(x_i\mid i\in I_w)^{\rm r\cal F}$
is a field of representatives for the residue field $\R(x_i\mid i\in I)^{\rm r\cal F}w$.
We apply Lemma~\ref{basic2} again, this time to the field $\R(x_i\mid i\in I_w)^{\rm r\cal F}$ with its
natural valuation $v$, to find that
\[
\R(x_i\mid i\in I_w)^{\rm r\cal F} \>=\> \R(\cX_w)^{\rm r\cal F} \>=\> \R(\cX_w)^{\rm h\cal F}\>.
\]
For coarsenings $w$ of $v$ and $w'$ of $w$ we have that $wa=0$ implies $w'a=0$, whence $\cX_w\subseteq\cX_{w'}\,$.
This yields eqation (\ref{ww'}) and concludes the proof.
\end{proof}
%
%
\section{An intrinsic version of ``truncation at 0''}\label{section4}
\begin{theorem}                             \label{thrprh}
Take $h\in LE_{\cal F}(x)$ such that $vh<0$. Then there are convex valuations $w_1<w_2<\ldots<w_k=v$ on
$LE_{\cal F}(x)$, $m_i\in\N$, monomials $d_{i,j}\in \cX$ and elements $c_{i,j}\in
LE_{\cal F}(x)w_i$, $1\leq i\leq k$, $1\leq j\leq m_i$, some $r_h\in\R$, and
$h^{^+}\in LE_{\cal F}(x)$ of value $vh^{^+}>0$, such that
\begin{equation}                            \label{rprh+}
h\>=\>c_{1,1}d_{1,1}+\ldots+c_{1,m_1}d_{1,m_1}+\ldots+c_{k,1}d_{k,1}+\ldots+c_{k,m_k}d_{k,m_k}
\,+\,r_h\,+\,h^{^+}
\end{equation}
with:
\sn
1) \ the values of the summands under the valuation $v$ are strictly increasing,
\sn
2) \ for each $i$, $1\leq i\leq k$,
\[
w_ic_{i,1}d_{i,1}\><\>\ldots\><\>w_ic_{i,m_i}d_{i,m_i}\>,
\]
and the values $vd_{i,j}$, $1\leq j\leq m_i\,$ generate an archimedean ordered subgroup of $vLE_{\cal F}(x)$,
\sn
3) \ for each $i$, $1\leq i\leq k-1$,
\[
c_{i+1,1}d_{i+1,1}+\ldots+c_{i+1,m_{i+1}}d_{i+1,m_{i+1}}+\ldots+c_{k,1}d_{k,1}+\ldots+c_{k,m_k}d_{k,m_k}
\]
lies in
$LE_{\cal F}(x)w_i\,$.
\pars
With these properties, the summands $c_{i,j}$, $d_{i,j}$ and the elements $r_h$ and $h^{^+}$ are uniquely determined.
\end{theorem}

Given the representation (\ref{rprh+}) of an element $h$ according to this theorem, the finite sum
\[
\mbox{\rm pp}(h)\>:=\> c_{1,1}d_{1,1}+\ldots+c_{1,m_1}d_{1,m_1}+\ldots+c_{k,1}d_{k,1}+\ldots+c_{k,m_k}d_{k,m_k}
\]
will be called the \bfind{principal part} of $h$; we set $\mbox{\rm pp}
(h):=0$ if $vh\geq 0$. The principal part is uniquely
determined once the set of monomials in $LE_{\cal F}(x)$ is fixed.
Note that $v(h-\mbox{\rm pp}(h)-r_h)>0$ with $r_h\in\R$.
\bn
The following lemma is the core of our proof:
\begin{lemma}                               \label{dense}
Let $(K,w)$ be a valued field with archimedean value group. Assume that $K=K_0(z_j\mid j\in J)$, where the values
$wz_j\,$, $j\in J$, are rationally independent and $w$ is trivial on $K_0\,$. Denote by $\cZ$
the multiplicative group $\langle z_j\mid j\in J \rangle$ generated by the elements $z_j\,$. Then the group ring
\[
R\>:=\> K_0[\cZ]
\]
lies dense in $K$ (with respect to the topology induced by $w$). Moreover, for each $a\in K\setminus \cO_w$ there
are uniquely determined elements $c_i\in K_0$ and $d_i\in \cZ$ with $wc_id_i<0$, $1\leq i \leq m$, such that
\begin{equation}                                    \label{=wsum}
a\,-\,\sum_{i=1}^m c_i d_i\,\in\,\cO_w\>.
\end{equation}
The same holds if we replace $K$ by its henselization or its completion.
\end{lemma}
\begin{proof}
In order to prove that $R$ lies dense in $K$ we have to show that for every $a\in K$ and every $\alpha\in
wK$ there is $a'\in R$ such that $w(a-a')>\alpha$.
Every $a\in K$ can be written as a quotient of two polynomials over $K_0$ in
finitely many of the $z_j\,$, that is,
\[
a\>=\> \frac{b'_1d'_1+\ldots+b'_k d'_k}{b''_1d''_1+ \ldots+b''_\ell d''_\ell}
\]
where $d'_1,\ldots,d'_k\in \cZ$ are distinct, $d''_1,\ldots,d''_\ell\in \cZ$ are distinct,
and $b'_i,b''_i\in K_0\setminus\{0\}$. From the rational independence of the values $wz_j$ it follows that every two
distinct elements in $\cZ$ and hence all $b''_i d''_i$ have distinct values. Therefore, we may assume that
$b''_1d''_1$ is the unique summand of least value in the denominator. We write
\[
b''_1d''_1+\ldots+b''_\ell d''_\ell\>=\>b''_1d''_1(1-d)\;\;\mbox{ with }\;
d\>:=\> -\frac{b''_2d''_2}{b''_1d''_1}-\ldots-\frac{b''_\ell d''_\ell}{b''_1d''_1}\>.
\]
Note that $\frac{d''_2}{d''_1},\ldots,\frac{d''_\ell}{d''_1}$ are elements of
$\cZ$ of positive value. Hence, also $wd>0$, and $w(1-d)=0$. It follows that
\[
w\left(\frac{1}{1-d}-\sum_{i=0}^{\ell} d^i\right)\>=\> w\left(1-(1-d)\sum_{i=0}^{\ell} d^i\right)
\>=\> w(-d^{\ell+1}) \>=\> (\ell+1)wd
\]
for every integer $\ell\geq 1$. Take $\alpha\in wK$. Since $wK$ is
archimedean, we can choose $\ell$ as big as to obtain that
\[
(\ell+1)wd\>\geq\>\alpha-w(b'_1d'_1+ \ldots +b'_kd'_k) (b''_1d''_1)^{-1}\>.
\]
For
\[
a'\>:=\>\left(\frac{b'_1d'_1}{b''_1d''_1}+\ldots +\frac{b'_kd'_k}{b''_1d''_1}\right) \sum_{i=0}^{\ell}
d^i\in R\;,
\]
this yields that
\[
w(a-a')\>=\> w(b'_1d'_1+\ldots +b'_kd'_k)(b''_1d''_1)^{-1}\left(\frac{1}{1-d}-\sum_{i=0}^{\ell} d^i\right)
\> \geq\>\alpha\>,
\]
showing that $R$ lies dense in $K$. Deleting all summands from $a'$ of value at least $\alpha$, we obtain a sum as
in (\ref{=wsum}) such that $wc_id_i<\alpha$ for all $i$ and $w(a-\sum_{i=1}^n c_i d_i)\geq\alpha$. For $\alpha=0$
this proves the existence of the elements $c_i\,$, $d_i$ as in the statement of the lemma. We have to prove their
 uniqueness.

\pars
Take two elements $r,r'\in R$ in which all summands have value smaller than $\alpha$, and such that $w(a-r)\geq
\alpha$ and $w(a-r')\geq\alpha$. It follows that $w(r-r')\geq\alpha$. Allowing the coefficients $b_i,c_i$ to be
zero, we can write $r=c_1d_1+\ldots+c_md_m$ and $r'=b_1d_1+\ldots+b_md_m$ where $d_1,\ldots,d_m\in \cZ$ are distinct
and $b_i,c_i\in K_0\,$. Then
\[
r'-r\>=\> (b_1-c_1)d_1+\ldots+(b_m-c_m)d_m\>.
\]
As the value of this sum is equal to the minimum of its summands $(b_i-c_i)d_i$, we see that $w(b_i-c_i)d_i
\geq w(r'-r) \geq\alpha$ for all $i$. But if there is some $i$ such that $b_i\ne c_i\,$, then this yields
$wd_i\geq\alpha$. As $b_i\ne 0$ or $c_i\ne 0$ it then follows that $wb_id_i=wd_i\geq\alpha$ or $wc_id_i=
wd_i\geq\alpha$, a contradiction to our initial assumption.
Consequently, the representation $r=c_1d_1+\ldots+c_md_m$ is uniquely determined when all $c_i$ are nonzero.

\pars
Every valued field is dense in its completion (by
definition). Since $wK$ is archimedean, the henselization of
$(K,w)$ lies in the completion and thus, $(K,w)$ is also dense in
its henselization. Since density is transitive, we find that $R$ is also
dense in the henselization and in the completion of $(K,w)$. It follows that the assertions we have proved for $a\in K$ also hold when $a$ lies in the henselization or completion.
\end{proof}

\pars
 The following is \cite[Lemma 6.41]{[KS]}:

\begin{lemma}                           \label{lfinite}
Let $x_i\in M$ such that $x_i>0$ and the values $vx_i\,$, $i\in I$ are
rationally independent. Then
\begin{equation}                            \label{finite}
\R(x_i\mid i\in I)^{\rm r\cal F}\;=\>
\bigcup_{I_0\subset I \>\rm finite}\;\;\bigcup_{k\in\N}^{}\;\;
\R(x_i^{1/k}\mid i\in I_0)^{\rm h{\cal F}}\>.
\end{equation}
\end{lemma}

\parm
In order to prove Theorem~\ref{thrprh}, take any $h\in LE_{\cal F}(x)$. We will work with the representation 
of $LE_{\cal F}(x)$ as given in Theorem~\ref{mth1}.
Lemma \ref{lfinite} shows that there is a finitely generated subgroup $\cX_h$ of $\cX$ such that $h\in
\R(\cX_h)^{\rm h{\cal F}}\subset \R(\cX_h)^{\rm r{\cal F}}$. Denote by $\cX'$ the divisible hull
of $\cX_h$ inside the divisible group $\cX$.
Since $v\cX'$ is isomorphic to $\cX'$ which is the divisible hull of a finitely generated abelian group,
it must have finite rational rank $\dim_{\Q} \Q\otimes v\cX'$. Therefore, $v\cX'$ has only finitely many convex subgroups, say,
\[
v\cX'\>=\>\Gamma_1\>\supset\>\Gamma_2\>\supset\>\ldots\>\supset\>\Gamma_{k+1}\>=\>\{0\}
\]
%
%
%
such that $\Gamma_i/\Gamma_{i+1}$ is archimedean ordered, for $1\leq i\leq k$.
Further, we choose convex valuations $w_1<\ldots< w_k$ on $LE_{\cal F}(x)$ such that the
restriction of $w_i$ to $K$ is a convex valuation corresponding to $\Gamma_{i+1}$, having value group $v\cX'/
\Gamma_{i+1}\,$. Since $\Gamma_{k+1}=\{0\}$, we can choose $w_k=v$. Each
\[
\cX'_i\>:= \> \{x'\in\cX'\mid vx'\in\Gamma_i\}\,,\quad 1\leq i\leq k
\]
is a $\Q$-sub vector space of the $\Q$-vector space $\cX'$.
We choose a $\Q$-basis $\cB_k$ of $\cX'_k$, and if $k>1$, $\Q$-bases $\cB_i$ of complements of
$\cX'_{i+1}$ in $\cX'_i$ for $1\leq i\leq k-1$. We obtain that $\cB:=\bigcup_i \cB_i$ is a $\Q$-basis of $\cX'$ and
that $h\in \R(\cB)^{\rm r{\cal F}}$.

From Lemma~\ref{lfinite} we infer that $h\in \R (b^{1/\ell}\mid b\in \cB)^{\rm h{\cal F}} =:K$ for some
$\ell\in\N$. Since we may replace each basis element $b$ by $b^{1/\ell}$, we can assume that $\ell=1$.

\pars
Now we proceed by induction on $k$. We assume that $k=1$ or that the theorem has been proven for all elements in
$\R (\bcB)^{\rm h{\cal F}}$, where $\bcB\subset\cB$ corresponds to a value group that has less convex subgroups than
$v\cX'$.

The value group of the convex valuation $w_1$ on $K$ is the archimedean ordered group $v\cX'/\Gamma_{k-1}\,$.
We set $\bcB=\emptyset$ if $k=1$, and $\bcB=\bigcup_{2\leq i\leq k} \cB_i$ if $k>1$. From
Lemma~\ref{basic2} we infer that $\R (\bcB)^{\rm h{\cal F}}$ is a
field of representatives for the residue field $Kw_1\,$.

We apply Lemma~\ref{dense} with $\cZ$ equal to the group generated by $\cB_1$ and $K_0=\R (\bcB)^{\rm h{\cal F}}$
to deduce the existence of uniquely determined elements
\[
c_{1,j}\,\in\, \R (\bcB)^{\rm h{\cal F}}\>\subset\> \R(\cX_{w_1})^{\rm h{\cal F}}
\>=\> LE_{\cal F}(x)w_1
\]
and
\[
d_{1,j}\,\in\, \cZ\>\subset\>\cX\>, \quad 1\leq j \leq m_1\>,
\]
with
\[
w_1c_{1,1}d_{1,1}\><\>\ldots\><\>w_ic_{1,m_1}d_{1,m_1} \><\>0
\]
and such that $w_1(h-\sum_{j=1}^{m_1} c_{1,j} d_{1,j})\geq 0$. Thus,
there is a unique element $\bar h\in \R (\bcB)^{\rm h{\cal F}}$ such that
\[
w_1(h-\sum_{i=1}^m c_i d_i-\bar h)> 0\>.
\]
By definition of $\cB_1$ we have that $v\cZ\subseteq\Gamma_1$ and $v\cZ\cap\Gamma_2=\{0\}$, which shows that
$v\cZ$ is archimedean. The same consequently holds for its subgroup that is generated by the values
$vd_{1,j}$, $1\leq j \leq m_1\,$.

\pars
If $k=1$, then $\R (\bcB)^{\rm h{\cal F}}=\R$ and we can set $r_h=\bar h\in \R$ to
obtain that $v(h-\sum_{i=1}^m c_i d_i -r_h)> 0$.

If $k>1$, then by induction hypothesis we know that our theorem holds for the element $\bar h$. We can thus write
\[
\bar h\>=\> c_{2,1}d_{2,1}+\ldots+c_{2,m_2}d_{2,m_2}+\ldots+c_{k,1}d_{k,1}+\ldots+c_{k,m_k}d_{k,m_k}
\,+\,r_{\bar h}\,+\,{\bar h}^{^+}
\]
such that the conditions of Theorem~\ref{thrprh} are satisfied for $\bar h$ in place of $h$ (and with $w_1$
omitted). Now we set $r_h:=r_{\bar h}$ and $h^{^+}:={\bar h}^{^+}$
to obtain a representation of the form
(\ref{rprh+}) for $h$. It is straightforward to see that properties 2) and 3) are satisfied. Also 1) holds since
$w_1 z\ne 0$ for all $z\in \cZ$, which implies that $vc_{1,j}d_{1,j}<\Gamma_2$ for $1\leq j\leq m_1\,$, whereas
$w_1 y=0$ for all $y\in \R (\bcB)^{\rm h{\cal F}}$, which implies that
$vc_{i,j}d_{i,j}\in\Gamma_2$ for $2\leq i\leq k$ and $1\leq j\leq m_i\,$. This completes the proof of Theorem~\ref{thrprh}.

%
\section{Applications}\label{section5}

\begin{theorem}                             \label{asppth}
Let $f,g:\R\rightarrow\R$ be ultimately positive
$\R$-definable functions. Then $f$ is asymptotic to $rg$ on
$\R$ for some positive $r\in\R$ if and only if the germs
$\log f$ and $\log g$ in $H(\R)$ have the same principal part.
\end{theorem}
\begin{proof}
We know from (\ref{asylog}) that $f$ is asymptotic to $rg$ on $\R$ if and
only if $v(\log f -\log rg)>0$. This in turn is equivalent to
$v(\log f -\log g)\geq 0$, since if the latter holds, then
there is some $r_0\in\R$ such that $v(\log f -\log g
-r_0)>0$, and we set $r=\exp r_0\,$. By the uniqueness of the principal
part, $v(\log f -\log g)\geq 0$ if and only if $\mbox{\rm pp}
(\log f)=\mbox{\rm pp}(\log g)$.
\end{proof}

To apply this theorem in the spirit of the Hardy problem, we take
${\cal F}$ to be any set of restricted analytic functions, closed
under partial derivations. Then by applying \cite[Theorem 6.30]{[KS]} 
simultaneously for ${\cal F}$ and ${\cal F}
_{\rm an}\,$, we find index sets $I_{\cal F}\subset I$ and elements
$x_i$ such that $LE_{{\cal F}}(x)=\R(x_i\mid i\in I_{\cal F})
^{\rm r\cal F}$ and $LE_{{\cal F}_{\rm an}}(x)=\R(x_i\mid i\in I)
^{{\rm r}{\cal F}_{\rm an}}$. So the monomials of $LE_{{\cal F}}(x)$
will also be monomials of $LE_{{\cal F}_{\rm an}}(x)$. Moreover, we
can take
\[
LE_{{\cal F}}(x)w \>\subseteq\> LE_{{\cal F}_{\rm an}}(x)w
\]
for each convex valuation $w$ and suitable $m_0\,$, according to
Theorem~\ref{mth1}. Using principal parts
determined by this choice of the $x_i$ and the residue fields, we get:

\begin{corollary}                           \label{asppcor}
Assume that $h:\R\rightarrow\R$ is definable in $\R_{\rm an, exp}$.
Then $\exp h$ is asymptotic to a composition of semialgebraic
functions, exp, log and restricted analytic functions in ${\cal F}$,
if and only if $\mbox{\rm pp}(h)\in LE_{{\cal F}}(x)$.
\end{corollary}

\pars
As an example, let us reconsider the Hardy problem. Here we assume in
addition that the $x_i$ include $x$ (cf.\ Theorem~\ref{mth1}).

Take two functions $f,g:\R\rightarrow\R$, definable in $\R_{\rm an,exp}$.
Assume that $\exp f(x)$ is asymptotic to $g(x)$, that is, $\lim_{x
\rightarrow\infty}\frac{\exp f(x)}{g(x)}=1$. This is equivalent to
$\lim_{x\rightarrow\infty}f(x)-h(x)=0$, where $h:(r,\infty) \rightarrow
\R$ for suitable $r\in\R$ is the function $\log g(x)$, which again is
definable in $\R_{\rm an, exp}$. This means that the function
$f(x)-h(x)$ is ultimately smaller than every nonzero constant function.
Equivalently, its germ $f-h$ in $H(\R_{\rm an, exp})$ is infinitesimal,
or in other words, $v(f-h)>0$.
\pars
As in \cite{[D--M--M2]}, let the function $i(x)$ denote the compositional
inverse of the function $x\log x$. Identifying $i(x)$ with its germ, we
have that $i(x)\in H(\R_{\rm an, exp})$. But by an argument about
Liouville extensions of the Hardy field $\R(x)$,
\cite[Corollary~4.6]{[D--M--M2]} shows that $i(x)\notin LE:=LE_{{\cal F_{\rm an}}}(x)$. Assume that $\exp i(x)$
were asymptotic to a function $g(x)$ which is a composition of
semialgebraic functions, exp and log. Through identification with its
germ, the latter means that $g(x)\in LE$. Then also $h(x):=\log g(x)\in
LE$, and $v(i(x)-h(x))>0$. Further, one shows as in \cite{[D--M--M2]} that
there is a convergent power series $f(X,Y)$ such that
\[i(x)\>=\>\frac{x}{\log x}\left(1+f\left(\frac{\log\log x}{\log x}\,,\,
\frac{1}{\log x}\right)\right)\;.\]
Now let $w$ be the convex valuation corresponding to the largest convex
subgroup not containing $vx$. This contains $v\log x$. Therefore, $w\log
x=0$ and $w\frac{\log x}{x}=-wx>0$. With
\begin{equation}                            \label{tildef}
\tilde{f}:=f\left(\frac{\log\log x}{\log x}\,,\,\frac{1}{\log x}\right)
\,\in\,\R(\log x,\log\log x)^{r{\cal F}_{\rm an}}\,\subseteq\,
LE_{{\cal F}_{\rm an}}(\log x)w\>,
\end{equation}
the representation of $i(x)$ is just
$i(x)=cx$, where $c=\frac{1}{\log x} (1+\tilde{f})\in H(\R_{\rm an,
exp})w$. Thus, $\mbox{\rm pp} (i(x))=i(x)\notin LE$. Hence by our
Corollary \ref{asppcor}, $\exp i(x)$ is not asymptotic to any element of $LE$.

\parb
Let us give a further application of Theorem~\ref{thrprh}. Denote by
${\cal L}_{\cal F}$ the language of ordered rings, enriched by symbols
for the functions from ${\cal F}$. Recall that every generalized power
series field $\R((G))$ has a canonical cross-section, sending $\alpha\in
G$ to the element $1_\alpha\in\R((G))$ which has a $1$ at $\alpha$ and
zeros everywhere else. ($1_\alpha$ is the characteristic function of the
singleton $\{\alpha\}$.)
\begin{corollary}                           \label{embps}
Take any ${\cal L}_{\cal F}$-embedding of $LE_{\cal F}(x)$ in
some generalized power series field $\R((G))$, and denote by $L$ its
image in $\R((G))$. Assume that the restriction of the canonical
cross-section of $\R((G))$ to $vL$ is a cross-section $\pi$ of $(L,v)$,
and that $L=\R(\pi vL)^{\rm r\cal F}$. Then the nonzero elements of the
support of each element in $L$ are bounded away from $0$.
\end{corollary}
\begin{proof}
For every convex valuation $w$ with associated convex subgroup $H_w
\subset G$, we have that $\R((G))w=\R((H_w))$.

Let $I\subset vL$ be a maximal set of rationally independent values.
Set $x_i:=1_{\alpha}$ for $i=\alpha\in I$. Then $\R(x_i\mid i\in I)
^{\rm r}=\R(\pi vL)^{\rm r}$ and hence, $\R(x_i\mid i\in I)
^{\rm r\cal F}=\R(\pi vL)^{\rm r\cal F}=L$ by hypothesis.
The monomials obtained from the $x_i$ are precisely the elements of
the form $r\cdot 1_{\alpha}$ with $r\in\R$ and $\alpha\in vL$. Note
that if $\alpha<H_w\,$, then for every $c\in \R((H_w))$, the support
of $c\,r1_{\alpha}$ is bounded away from $0$ by every element $\beta$
which satisfies $\alpha+H_w<\beta<0$. For example, $\beta= \alpha/2$
is a good choice.

Take $h\in L$ and consider the representation (\ref{rprh+}) with
respect to the monomials $x_i$ and the residue fields $\R((H_w))$. Now
$\mbox{support}(h)\setminus\{0\}$ is the union of the support of $c_1d_1
+\ldots+c_md_m$ and the support of $h^{^+}$. The latter is bounded away
from 0 by $vh^{^+}$. The support of $c_1d_1+ \ldots+c_md_m$ is the union
of the supports of $c_1d_1,\ldots,c_md_m\,$. This union is bounded away
from 0 by $\frac{1}{2}vd_m\,$.
\end{proof}
\n
Note that the embeddings of $H(\R_{\rm an,exp})$ and of $LE$ in the
logarithmic power series field $\R((t))^{LE}$ given in \cite{[D--M--M2]}
satisfy the conditions of the corollary. (Recall that $\R((t))^{LE}$
can be viewed as a subfield of a suitable power series field.)
\bn
\bn
\bn

\end{document}